\documentclass{article}
\usepackage{graphicx} 
\usepackage{amsthm, amssymb, amsmath, enumerate, bm} 
\usepackage{hyperref} 
\usepackage[dvipsnames]{xcolor} 
\usepackage{tocloft} 

\cftsetindents{section}{0em}{2em}
\cftsetindents{subsection}{0em}{2em}

\makeatletter
\g@addto@macro{\UrlBreaks}{%
\do\/\do\a\do\b\do\c\do\d\do\e\do\f%
\do\g\do\h\do\i\do\j\do\k\do\l\do\m%
\do\n\do\o\do\p\do\q\do\r\do\s\do\t%
\do\u\do\v\do\w\do\x\do\y\do\z%
\do\A\do\B\do\C\do\D\do\E\do\F\do\G%
\do\H\do\I\do\J\do\K\do\L\do\M\do\N%
\do\O\do\P\do\Q\do\R\do\S\do\T\do\U%
\do\V\do\W\do\X\do\Y\do\Z}

\newcommand{\@endstuff}{\par\vspace{\baselineskip}\noindent\small
\begin{tabular}{@{}l}\scshape Tianyi Feng (corresponding author), University of St Andrews, Scotland\\\textit{E-mail address:} \href{mailto:tf66@st-andrews.ac.uk}{\color{OliveGreen}{\texttt{tf66@st-andrews.ac.uk}}}\end{tabular} \\

\begin{tabular}{@{}l}\scshape Jonathan M. Fraser, University of St Andrews, Scotland\\\textit{E-mail address:} \href{mailto:jmf32@st-andrews.ac.uk}{\color{OliveGreen}{\texttt{jmf32@st-andrews.ac.uk}}}\end{tabular}}
\AtEndDocument{\@endstuff}
\makeatother

\renewcommand{\geq}{\geqslant}
\renewcommand{\leq}{\leqslant}
\newcommand{\leqs}{\lesssim}
\newcommand{\geqs}{\gtrsim}

\renewcommand{\epsilon}{\varepsilon}

\newcommand{\calpha}{C^\alpha [0,1]}
\newcommand{\lilpha}{c^\alpha [0,1]}
\newcommand{\cmega}{C^\omega[0,1]}
\newcommand{\lilmega}{c^\omega[0,1]}
\newcommand{\bb}[1]{\mathbb{#1}}

\newcommand{\fx}[2]{(#2 , #1 (#2))}
\newcommand{\tilf}{\tilde{f}}
\newcommand{\F}[2]{\frac{|f(#1)-f(#2)|}{|#1-#2|^\alpha}}

\newcommand{\Fo}[2]{\frac{|f(#1)-f(#2)|}{\omega(|#1-#2|)}}
\newcommand{\tFo}[2]{\frac{|\tilde{f}(#1)-\tilde{f}(#2)|}{\omega(|#1-#2|)}}
\newcommand{\Ho}[2]{\frac{|h(#1)-h(#2)|}{\omega(|#1-#2|)}}

\newcommand{\floor}[1]{\left\lfloor #1 \right\rfloor}
\newcommand{\ad}{\dim_{\textup{A}}}

\newcommand{\ubd}{\overline{\dim}_{\textup{B}}}
\newcommand{\epk}{\epsilon_k}

\newtheorem{thm}{Theorem}[section]
\newtheorem{lma}[thm]{Lemma}
\newtheorem{cor}[thm]{Corollary}
\newtheorem{prop}[thm]{Proposition}
\newtheorem{ques}[thm]{Question}
\theoremstyle{definition}

\numberwithin{equation}{section} 

\title{The Assouad spectrum and dimension of typical graphs}
\renewcommand\footnotemark{}
\author{Tianyi Feng and Jonathan M. Fraser \thanks{The authors were financially supported by a \emph{Leverhulme Trust Research Project Grant} (RPG-2023-281). JMF was also funded by \emph{EPSRC Open Fellowship} (EP/Z533440/1).}}
\date{}

\begin{document}

\maketitle

\begin{abstract}
We investigate the Assouad spectrum and dimension  of    graphs of  functions lying in certain Banach spaces.  We find the typical values in the sense of Baire category, proving that these values are often as large as possible, given the constraints of the particular function space.  For example, we demonstrate that in the little $\alpha$-Hölder spaces, a typical graph has Assouad dimension $2$ and Assouad spectrum $\min\{2,2-\frac{\alpha-\theta}{1-\theta}\}$; whereas   in the space associated with modulus of continuity $t(1+|\log t|)$, a typical graph   has Assouad dimension $2$ but quasi-Assouad dimension (and Assouad spectrum) equal to $1$.
\\ \\ 
\emph{Mathematics Subject Classification}: primary: 26A21, 28A80 ; secondary: 26A15, 26A16.
\\
\emph{Key words and phrases}: Baire category, Assouad dimension, quasi-Assouad dimension, Assouad spectrum, modulus of continuity, Hölder functions, graphs of continuous functions.
\end{abstract}

\tableofcontents

\section{Introduction}

\subsection{Overview}

The fractal dimensions of the graphs of continuous functions have long been a subject of interest with many particular examples playing a prominent role, such as the Weierstrass function, the Takagi function, or random processes such as Brownian motion.  Computing the dimensions of a specific graph is often very challenging and an alternative approach is to instead consider the `generic answer' within a specific function space.  This is the viewpoint of this paper and we adopt the Baire category approach and study the typical dimensions of graphs in this context.  This problem is also well-studied, but not in the setting of the Assouad spectrum which is a family of dimensions which interpolates between the upper box dimension and the (quasi-)Assouad dimension.  We find that the extra information coming from the Assouad spectrum is useful in providing a  nuanced answer to such questions.

 Denote the set of real-valued continuous functions on $[0,1]$ by $C^0[0,1]$ and the infinity norm by $||\cdot||_\infty$. Write $G_f = \{(x,f(x)) \ | \ x \in [0,1]\}\subseteq \bb{R}^2$ for the graph of a function $f: [0,1] \to \bb{R}$. It has been proven in \cite{PACKING, box, Haus} that a (Baire) typical $f \in (C^0[0,1], ||\cdot||_\infty)$ satisfies
\[
{\dim_{\textup{H}}} G_f = {\underline{\dim}_{\textup{B}}} G_f = 1 < 2 = {\dim_{\textup{P}}} G_f = \ubd G_f,
\]
where ${\dim_{\textup{H}}} G_f, {\underline{\dim}_{\textup{B}}} G_f, {\dim_{\textup{P}}} G_f$, and $\ubd G_f$ refer to the Hausdorff, lower box, packing, and upper box dimensions of $G_f$ respectively. See \cite{KJF} for precise definitions and properties of these dimensions. Observe that $G_f$ has full packing and upper box dimensions, in the sense that they are equal to the dimension of the ambient space $\bb{R}^2$. 

The focus of this paper will be on the \emph{Assouad dimension} and the \emph{Assouad spectrum}. Since the Assouad dimension and spectra of a set are at least its upper box dimension, it follows trivially that the Assouad dimension and spectra of a typical continuous graph are full as well. However, this does not necessarily apply to functions under more specific continuity constraints. For example, given $\alpha \in (0,1)$, a function $f$ is $\alpha$-H\"older continuous if 
\[
\sup_{x,y \in [0,1]} \F{x}{y} < \infty
\]
and a result in \cite{KJF} indicates that $\ubd G_f \leq 2-\alpha < 2$ for all such functions. Despite this, we cannot obtain any non-trivial bound for the Assouad dimension. In fact, we shall see below that a typical function in the little Hölder space, defined below in Section \ref{spaceDef}, has a graph with Assouad dimension 2; see Theorem \ref{adTyp}. Since the Assouad spectrum lives in between the box and Assouad dimensions, it is then an interesting question to consider the Assouad spectrum of a typical graph. In particular, we will show that a typical graph in the little Hölder space has Assouad spectrum $2-\frac{\alpha-\theta}{1-\theta}$ for $\theta \leq \alpha$; see Theorem \ref{specTyp}.

The paper is structured as follows: Section 1 is dedicated to background on function spaces, Assouad dimension, Baire category,  and other relevant preliminaries. The main results will be presented in Section 2 with proofs in Section 3.

\subsection{Our function spaces and some notation} \label{spaceDef}

Recall that $\omega:[0,1] \to [0,\infty]$ is a \emph{modulus of continuity} if $\omega$ is increasing, $\omega(t)>0$ for $t>0$ and satisfies
\[
\lim_{t \to 0} \omega(t) = \omega(0) = 0.
\]
Throughout the paper, by increasing we mean non-decreasing, i.e. $\omega(x) \leq \omega(y)$ for $x \leq y$. For $f:[0,1] \to \bb{R}$ we define 
\[
[f]_\omega = \sup_{\substack{x,y \in [0,1] \\ x \neq y}} \frac{|f(x)-f(y)|}{\omega(|x-y|)},
\]
and denote the set of \emph{$\omega$-continuous functions} by
\[
C^\omega[0,1] =\{f: [0,1] \to \mathbb{R} \ \ | \ \ [f]_\omega < \infty\}.
\]
One can show that $||\cdot||_\omega: C^\omega[0,1] \to \bb{R}$ where
\[
||f||_\omega = ||f||_\infty +  [f]_\omega
\]
defines a norm on $C^\omega[0,1]$, which we refer to as the \emph{$\omega$-norm}. We call $(\cmega, ||\cdot||_\omega)$ the \emph{$\omega$-space}. We are particularly interested in the subset of functions that are relatively smooth locally, that is, all the functions $f \in \cmega$ satisfying
\[
\lim_{t \to 0} \sup \bigg\{\frac{|f(x)-f(y)|}{\omega(|x-y|)}: x,y \in [0,1], \ x \neq y, \ |x-y| \leq t\bigg\} = 0.
\]
We write $\lilmega$ for the set consisting of such functions and refer to $(\lilmega, ||\cdot||_\omega)$ as the \emph{little $\omega$-space}. 

Given $\alpha \in (0,1)$, if we take $\omega(t) = t^\alpha$ then $\cmega$ is the set of $\alpha$-Hölder functions. We refer to the corresponding space as the $\alpha$-Hölder space and $(\lilmega, ||\cdot||_\omega)$ as the little Hölder space. Another prominent class of continuous functions are the \emph{Lipschitz functions}, where we say a function $f: [0,1] \to \bb{R}$ is Lipschitz if 
\[
\sup_{x,y \in [0,1]} \frac{|f(x)-f(y)|}{|x-y|} < \infty.
\]
Taking $\omega(t)=t$ above, we get the space of Lipschitz functions $C^1[0,1]$, or the Lipschitz space. 

Here we point out that although each $\omega$ determines a unique $\omega$-norm and hence a unique $\omega$-space, it does not uniquely determine $\cmega$. For example, for $\omega_1(t) = t$ and $\omega_2(t)=2t$, both $C^{\omega_1}[0,1]$ and $C^{\omega_2}[0,1]$ are the set of all Lipschitz functions. 

For $d \in \bb{N}$, $x \in \bb{R}^d$ and $r>0$, denote the closed ball of radius $r$ centred at $x$ by $B(x,r) = \{y \in \bb{R}^d: |x-y|\leq r\}$ and the open counterpart by $B^o(x,r) = \{y \in \bb{R}^d: |x-y| < r\}$. Similarly, we write $Q(x,r)$ for the (closed) cube with sides parallel to the axes and of length $r$ centred at $x$. Furthermore, when the context is clear, we write $A \geqs B$ if there exists a $c>0$ such that $A \geq cB$, and $A \leqs B$ if $B \geqs A$. Finally, we use $A \approx B$ when $A \geqs B$ and $B \geqs A$.

\subsection[\texorpdfstring{Baire category and the space $\lilmega$}
                        {Baire category and the space lilmega}]
        {Baire category and the space $\lilmega$} 
In a complete metric space $X$, a set is \emph{meagre} if it is a countable union of nowhere dense subsets of $X$. The complement of a meagre set is called \emph{residual}. We say a \emph{typical} element has a certain property if the set of elements which have the property is residual. One may think of typical elements as generic in a topological sense. For more details see \cite{oxBaire}.

In order to investigate typical functions in $(\lilmega, ||\cdot||_\omega)$, it is necessary for it to be a complete metric space. For completeness, we include the proof of the following lemma.

\begin{lma}
$(\cmega, ||\cdot||_\omega)$ and $(\lilmega, ||\cdot||_\omega)$ are complete.    
\end{lma}
\begin{proof}
Let $(f_n) \subseteq (\cmega, ||\cdot||_\omega)$ be Cauchy, in which case $(f_n)$ is also Cauchy in $(C^0[0,1], ||\cdot||_\infty)$.  Since $(C^0[0,1], ||\cdot||_\infty)$ is complete, there exists some $f \in C^0[0,1]$ such that $||f_n-f||_\infty \to 0$. Let $\epsilon>0$. By taking subsequences and relabelling the terms, we may assume $(f_n)$ satisfies $||f_n-f_{n+1}||_\omega \leq 2^{-n} \epsilon$ and $||f_n-f||_\infty \leq 2^{-n-1}\omega(2^{-n})\epsilon$ for all $n$. Choose $N$ such that $||f_n-f||_\infty \leq \epsilon$ for all $n \geq N$. Consider $f_n$ where $n \geq N$. Let $x,y \in [0,1]$ be distinct and let $m$ be the integer such that $2^{-m} \leq |x-y| < 2^{-m+1}$. If $m \leq n$,
\[
\frac{|(f_n-f)(x)-(f_n-f)(y)|}{\omega(|x-y|)} \leq \frac{2||f_n-f||_\infty}{\omega(2^{-n})} \leq 2^{-n}\epsilon \leq \epsilon.
\]
Otherwise $m > n$ and
\begin{align*}
\frac{|(f_n-f)(x)-(f_n-f)(y)|}{\omega(|x-y|)} &\leq   \frac{|(f_m-f)(x)-(f_m-f)(y)| + \sum^{m-1}_{k=n} |(f_k-f_{k+1})(x)-(f_k-f_{k+1})(y)|}{\omega(|x-y|)} \\
&\leq  \frac{|(f_m-f)(x)-(f_m-f)(y)|}{\omega(2^{-m})} +  \sum^{m-1}_{k=n} ||f_k-f_{k+1}||_\omega  \\
&\leq \sum^{m}_{k=n} 2^{-k} \epsilon \\
&\leq \epsilon.
\end{align*}
In other words, $[f_n-f]_\omega \leq \epsilon$. Hence, 
\begin{gather*}
||f_n-f||_\omega = ||f_n-f||_\infty + [f_n-f]_\omega \leq 2\epsilon, \\
[f]_\omega \leq [f_n-f]_\omega + [f_n]_\omega < \infty
\end{gather*}
and so $f_n \to f \in (\cmega, ||\cdot||_\omega)$ also. This shows $(\cmega, ||\cdot||_\omega)$ is complete. 

Now, let $(f_n) \subseteq \lilmega$ and suppose $f_n \to f \in \cmega$ w.r.t.$||\cdot||_\omega$. Again by taking subsequences and relabelling the terms, we may assume $||f_n-f||_\omega \leq n^{-1}$ for all $n$. Define $(t_n)$ to be a strictly decreasing sequence such that $t_n>0$ and
\[
\frac{|f_n(x)-f_n(y)|}{\omega(|x-y|)} \leq \frac{1}{n}
\]
for all $x,y \in[0,1]$ distinct with $|x-y| \leq t_n$, using the fact that $f_n \in \lilmega$. Let $0<t \leq t_1$, then $t_{n+1} < t \leq t_n$ for some integer $n$. Suppose $x,y \in[0,1]$ are distinct with $|x-y| \leq t$. Then,
\begin{align*}
\Fo{x}{y} &\leq  \frac{|(f-f_n)(x)-(f-f_n)(y)|}{\omega(|x-y|)} + \frac{|f_n(x)-f_n(y)|}{\omega(|x-y|)}  \\
&\leq ||f_n-f||_\omega + \frac{|f_n(x)-f_n(y)|}{\omega(|x-y|)}\\
&\leq \frac{1}{n} + \frac{1}{n} \\
&= \frac{2}{n}\\
&\to 0
\end{align*}
as $t \to 0$. This shows $f \in \lilmega$. Thus, $\lilmega$ is closed under $||\cdot||_\omega$ in $\cmega$, from which the completeness of $(\lilmega, ||\cdot||_\omega)$ follows.
\end{proof}

It may seem that $\lilmega$ is rather small compared to the bigger set $\cmega$, indeed Berezhnoi shows in \cite{Berezhnoi} that for $\omega$ concave $\lilmega$ is meagre in $(\cmega, ||\cdot||_\omega)$. However, if we consider the set $\calpha$ of $\alpha$-H\"older functions and its little set $\lilpha$, we see that $C^{\alpha'}[0,1] \subseteq \lilpha$ for all $0 < \alpha < \alpha' < 1$.

\subsection{Assouad dimension and spectrum}
 Let $E \subseteq \bb{R}^d$ be non-empty and $r>0$. We say a collection of balls $\{B(x_i,r)\}_i$ is an \emph{$r$-packing} of $E$ if $x_i \in E$ for all $i$ and the balls are pairwise disjoint. Denote the maximum number of balls in an $r$-packing of $E$ by $N_r(E)$. Recall that the \emph{upper box dimension} of a non-empty bounded set $F \subseteq \bb{R}^d$ is given by
\[
\ubd F = \limsup_{r \to 0} \frac{\log N_r(F)}{-\log r}.
\]
For a non-empty set $F \subseteq \bb{R}^d$, its \emph{Assouad dimension} is defined as
\begin{align*}
\ad F = \inf \biggl\{ \beta:& \text{ there exists } C>0 \text{ such that for all } 0<r<R \text{ and } x \in F, \\
                      &  N_r(B^o(x,R) \cap F) \leq C \biggl(\frac{R}{r}\biggr)^\beta \biggr\}.    
\end{align*}

It can be shown that $\ubd F \leq \ad F$ for all $F \subseteq \bb{R}^d$ bounded, where the inequality can be strict. The Assouad spectrum then serves to interpolate between the upper box dimension and the Assouad dimension. More precisely, we define the \emph{Assouad spectrum} of $F$ to be
\begin{align*}
\ad^{\theta} F = \inf \biggl\{ \beta:& \text{ there exists } C>0 \text{ such that for all } 0<r<1 \text{ and } x \in F, \\
                      &  N_r(B^o(x,r^\theta) \cap F) \leq C \biggl(\frac{r^\theta}{r}\biggr)^\beta \biggr\}    
\end{align*}
for $\theta \in (0,1)$. This was introduced by Fraser and Yu in \cite{FrYu2, FrYu1}. The \emph{quasi-Assouad dimension} $\dim_{\textup{qA}} F $ was defined by Lü and Xi in \cite{Quas} and it may be defined by $\dim_{\textup{qA}} F = \lim_{\theta \to 1} \ad^{\theta} F $, see \cite{JMF_book}. 

In all definitions above, $N_r(B^o(x,R) \cap F)$ could be replaced with $N'_r(Q(x,R) \cap F)$, where $N'_r(E)$ is the minimum number of (closed) cubes which $E$ intersects in the standard $r$-mesh oriented at the origin. 

Note that $\ubd F \leq \ad^{\theta} F\leq  \dim_{\textup{qA}} F \leq \ad F$ holds for all $\theta \in (0,1)$, and in fact
\[
\ad^{\theta} F\leq \min \left\{ \frac{\ubd F}{1-\theta}, \ \dim_{\textup{qA}} F\right\}.
\]
Moreover, the mapping $\theta \mapsto \ad^{\theta}F$ is continuous in $\theta \in (0,1)$ and
\begin{align*}
\ad^{\theta} F &\to \ubd F \ \text{ as } \ \theta \to 0.
\end{align*}
See \cite{JMF_book} for more discussions on the Assouad dimension and spectrum.

\section{Main results}
We state our main results in this section and defer our proofs to Section \ref{secProofs}. We begin with the Assouad spectra of the graphs, and obtain a general upper bound which holds for all functions. This estimate is a mild generalisation of a recent estimate of Chrontsios-Garitsis and Tyson \cite{holder}, which considered the space of Hölder functions.

\begin{prop} \label{omgBound}
Let $\omega$ be a modulus of continuity such that $\omega(1) < \infty$ and 
\[
\eta = \liminf_{t \to 0} \frac{\log\omega(t)}{\log t}
\]
satisfies $\eta \leq 1$. Then, for all $f \in \cmega$ and $\theta \in (0,\eta)$,
\[
\ad^{\theta} G_f \leq 2-\frac{\eta-\theta}{1-\theta} < 2.
\]
Additionally, we have $\ubd G_f \leq 2-\eta$. 
\end{prop}

See Section \ref{specProofs} for the proof. Indeed, if $\eta >1$ then the only functions in $\cmega$ are constant, whose graphs have Assouad dimension and spectra equal to $1$, see Proposition \ref{omgCnst}. Notice also that if we simply compare the Assouad spectrum to the upper box dimension using the general inequality between them, we obtain
\[
\ad^{\theta} G_f \leq \frac{\ubd G_f}{1-\theta} \leq \frac{2-\eta}{1-\theta}.
\]
This is weaker than the upper bound in Proposition \ref{omgBound}, as $2-\frac{\eta-\theta}{1-\theta} <\frac{2-\eta}{1-\theta}$ for $\theta \in (0, \eta)$. 

For $\alpha \in (0,1)$, the $\alpha$-Hölder spaces have $\eta=\alpha$. In the recent paper \cite{holder} by Chrontsios-Garitsis and Tyson on the Assouad spectrum of Hölder graphs, the upper bound given in \cite[Theorem 1.1]{holder} coincides with the one above. For the upper box dimension, the upper bound also agrees with the one given in \cite[Corollary 11.2]{KJF}.

Next, we give a technical lemma for concave moduli. The proof can be found in Section \ref{proofPrelim}.

\begin{lma} \label{omgCcvLim}
Let $\omega$ be a concave modulus of continuity, then the limit
\[
\lim_{t \to 0} \frac{\log \omega(t)}{\log t}
\]
exists.
\end{lma}

The proof of the following theorem is presented in Section \ref{specProofs}.

\begin{thm} \label{specTyp}
Let $\omega$ be a concave modulus of continuity and write
\[
\eta = \lim_{t \to 0} \frac{\log\omega(t)}{\log t}.
\]
Then
\begin{enumerate}[(i)]
\item If $\eta \geq 1$, then $\ad^{\theta} G_f= 1$ for all $f \in \cmega$ and $\theta \in (0,1)$. 
\item If $\eta < 1$, then for a typical $f \in (\lilmega, ||\cdot||_\omega)$,
\[
\ad^{\theta} G_f = 2- \frac{\eta-\theta}{1-\theta}
\]
for $\theta \leq \eta$, and $\ad^{\theta} G_f = 2$ for $\theta > \eta$. In particular, this holds in all little Hölder spaces with $\eta = \alpha \in (0,1)$.
\end{enumerate}
\end{thm}

In particular, part (ii) shows that the upper bound for the Assouad spectrum in Proposition \ref{omgBound} is sharp, which agrees with \cite[Theorem 1.2]{holder}. 

An immediate consequence of Theorem \ref{specTyp} is the following result describing the quasi-Assouad dimension.

\begin{cor} \label{qAtyp}
Let $\omega$ be a concave modulus of continuity and write
\[
\eta = \lim_{t \to 0} \frac{\log\omega(t)}{\log t}.
\]
Then
\begin{enumerate}[(i)]
\item If $\eta \geq 1$, then $\dim_{\textup{qA}} G_f= 1$ for all $f \in \cmega$. 
\item If $\eta < 1$, then for a typical $f \in (\lilmega, ||\cdot||_\omega)$,
\[
\dim_{\textup{qA}} G_f = 2.
\]
\end{enumerate}
\end{cor}

Next we present a similar result regarding the Assouad dimension of the graphs, the proof of which can be found in Section \ref{adProofs}.

\begin{thm} \label{adTyp}
Suppose $\omega$ is a concave modulus of continuity.
\begin{enumerate}[(i)]
\item If $\omega$ satisfies 
\[
\lim_{t \to 0} \frac{\omega(t)}{t} < \infty, 
\]
then $\cmega$ is the set of all Lipschitz functions and $\ad G_f = 1$ for all $f \in \cmega$. 

\item Otherwise $\omega$ is such that
\[
\lim_{t \to 0} \frac{\omega(t)}{t} = \infty, 
\]
and then for a typical $f \in (\lilmega, ||\cdot||_\omega)$ we have $\ad G_f = 2$. In particular, this holds in all little Hölder spaces.
\end{enumerate}
\end{thm}

The subtle yet dramatic difference between the Assouad and quasi-Assouad dimensions can be illustrated via the above results. Consider the modulus of continuity given by
\[
\omega(t) =t(1+|\log t|),
\]
 noting that  $\omega$ is concave  and
\[
\lim_{t \to 0} \frac{\log\omega(t)}{\log t}=1 \ \text{ and } \lim_{t \to 0} \frac{\omega(t)}{t} = \infty.
\]
By Corollary \ref{qAtyp} and Theorem \ref{adTyp}, a typical $f \in (\lilmega, ||\cdot||_\omega)$ has $\dim_{\textup{qA}} G_f = 1$ but $\ad G_f = 2$.

Finally, it is natural to ask what happens in the bigger space $(\cmega,||\cdot||_\omega)$. Without the assumption that the functions are locally flat, it might seem intuitive that the more a function oscillates the more likely it would have a greater Assouad dimension and spectrum. However, the construction used to prove Theorem \ref{specTyp} (ii) and Theorem \ref{adTyp} (ii) does not work here.  The problem lies in the part of the proof where we exhibit density.  For example, if a function in $\cmega$ was to oscillate maximally but somehow have Assouad dimension strictly less than 2, then we do not know how to approximate it in $||\cdot||_\omega$ by a function with Assouad dimension 2, since the change in norm of the approximating functions would need to be  large. We therefore pose the following questions.

\begin{ques}
For a concave $\omega$ with 
\[
\eta = \lim_{t \to 0} \frac{\log\omega(t)}{\log t}<1,
\]
 is it true that a typical $f \in (\cmega,||\cdot||_\omega)$ satisfies  
\[
\ad^{\theta} G_f = 2- \frac{\eta-\theta}{1-\theta}
\]
for $\theta \leq \eta$? 
\end{ques}
\begin{ques}
For a concave $\omega$ with
\[
\lim_{t \to 0} \frac{\omega(t)}{t} = \infty, 
\]
is it true that a typical $f \in (\cmega,||\cdot||_\omega)$ satisfies $\ad G_f = 2$?
\end{ques}

\section{Proofs} \label{secProofs}
In this section we prove the results stated in Section 2. 

\subsection{Preliminaries} \label{proofPrelim}
Although the paper concerns largely with functions with concave modulus of continuity, we begin with a case where $\omega$ is not necessarily concave.

\begin{prop} \label{omgCnst}
Let $\omega$ be a modulus of continuity satisfying    
\[
\eta = \limsup_{t \to 0} \frac{\log\omega(t)}{\log t} >1.
\]
Then $\cmega$ contains only constant functions.
\end{prop}
\begin{proof}
Let $f \in \cmega$ and let $x,y \in [0,1]$ be distinct, where we assume without loss of generality that $x<y$. Let $\delta \in (0, \eta-1)$, and let $(t_k) \subseteq [0,1]$ be decreasing with $t_1 \leq |x-y|$ and $t_k \to 0$ such that
\[
\frac{\log \omega(t_k)}{\log t_k} > 1+\delta
\]
for all $k$, which rearranges to
\[
\omega(t_k) < t_k^{1+\delta}.
\]
Choose $N_k \in \bb{N}$ so that
\[
\frac{t_k}{2}\leq \frac{|x-y|}{N_k} \leq t_k.
\]
Now, let $x=x_0<x_1<...<x_{N_k}=y$ be evenly spaced. Because $f \in \cmega$, the triangle inequality and the monotonicity of $\omega$ yield
\begin{align*}
|f(x)-f(y)| &\leq |f(x)-f(x_1)| + ... + |f(x_{N_k-1})-f(y)| \\
&\leq [f]_\omega \sum_{i=1}^{N_k} \omega(|x_{i-1}-x_i|) \\
&= [f]_\omega  \sum_{i=1}^{N_k} \omega\bigg(\frac{|x-y|}{N_k}\bigg) \\
&\leq [f]_\omega \cdot N_k \cdot \omega(t_k)  \\
&\leq [f]_\omega \cdot N_k \cdot t_k^{1+\delta} \\
&\leq [f]_\omega \cdot 2|x-y| \cdot  t_k^\delta \\
&\to 0
\end{align*}
as $k \to \infty$. Thus, $f(x)=f(y)$. Since $x, y \in [0,1]$ were arbitrarily chosen, we conclude that $f$ is in fact constant.     
\end{proof}

Next we introduce a technical lemma concerning concave moduli of continuity which will be used in several proofs below.

\begin{lma} \label{linIneq2}
Let $\omega$ be a concave modulus of continuity. If $f \in C^0[0,1]$ is affine over some $[a,b] \subseteq [0,1]$ where $a<b$, then
\[
\sup_{\substack{x,y \in [a,b] \\ x \neq y}} \frac{|f(x)-f(y)|}{\omega(|x-y|)} = \frac{|f(a)-f(b)|}{\omega(|a-b|)}.
\]
\end{lma}
\begin{proof}
Let $x,y \in [a,b]$ be distinct. We may write $x = a+\gamma_1(b-a)$ and $y = a+\gamma_2(b-a)$ for some $\gamma_1,\gamma_2 \in [0,1]$. Since $f$ is affine,
\[
\frac{|f(x)-f(y)|}{\omega(|x-y|)} = \frac{|\gamma_1-\gamma_2||f(a)-f(b)|}{\omega(|\gamma_1-\gamma_2||a-b|)} \leq \frac{|f(a)-f(b)|}{\omega(|a-b|)},
\]
where $|\gamma_1-\gamma_2|\omega(|a-b|) \leq \omega(|\gamma_1-\gamma_2||a-b|)$ by the concavity of $\omega$.
\end{proof}

\subsection[\texorpdfstring{Assouad spectrum of typical little $\omega$ graphs}
                        {Assouad spectrum of typical little omega graphs}]
        {Assouad spectrum of typical little $\omega$ graphs} \label{specProofs}
        
In this section we consider mainly the Assouad spectrum. First we consider the case where $\omega$ is any modulus of continuity.

\begin{proof}[Proof of Proposition \ref{omgBound}]
First, let $\theta < \eta$ and $0<\epsilon < (\eta-\theta)/(1-\theta)$. Choose $r_0>0$ such that for all $r \in (0,r_0]$ 
\begin{equation} \label{omgEq1}
\frac{\log\omega(r)}{\log r} - \theta \geq \epsilon(1-\theta),    
\end{equation}
from which we obtain
\[
\frac{\omega(r)}{r} \leq \bigg(\frac{r^\theta}{r}\bigg)^{1-\epsilon}    
\]
after rearranging. If $r_0 \leq r < 1$, then
\[
\omega(r)r^{\epsilon(\theta-1)-\theta} \leq \omega(1)r_0^{\epsilon(\theta-1)-\theta}.
\]
Hence
\begin{equation} \label{omgEq2}
\frac{\omega(r)}{r} \leq C\bigg(\frac{r^\theta}{r}\bigg)^{1-\epsilon}    
\end{equation}
for all $0<r<1$, where $C = \max \{1, \ \omega(1)r_0^{\epsilon(\theta-1)-\theta} \}< \infty$. Let $f \in \cmega$ and $0<r<1$. Let $x \in [0,1]$ and write $\bm{x} = (x,f(x))$. The number of $r$-columns of the form $[nr,(n+1)r] \times \bb{R}$ for some integer $n$ which the square $Q(\bm{x},r^\theta)$ intersects is at most $r^{\theta-1}+2$. Since 
\[
|f(y) - f(z)| \leq [f]_\omega\omega(|y-z|)\leq [f]_\omega \omega(r)
\]
for all $y,z \in [nr,(n+1)r]$, the number of $r$-squares $G_f$ intersects in each $r$-column is at most $\frac{[f]_\omega \omega(r)}{r} +2$. Therefore
\begin{align*}
N_r'(Q(\bm{x},R) \cap G_f) &\leq \bigg(\frac{r^\theta}{r} + 2\bigg) \bigg(\frac{[f]_\omega \omega(r)}{r} +2\bigg)   \\
&\leqs \frac{r^\theta}{r} \cdot \frac{\omega(r)}{r} \\
&\leqs \frac{r^\theta}{r}\cdot \bigg( \frac{r^\theta}{r}\bigg)^{1-\epsilon} \\
&= \bigg( \frac{r^\theta}{r}\bigg)^{2-\epsilon}
\end{align*}
by (\ref{omgEq2}), which shows $\ad^{\theta} G_f \leq 2-\epsilon$. Since we chose $\epsilon < (\eta-\theta)/(1-\theta)$, we obtain $\ad^{\theta} G_f \leq 2-\frac{\eta-\theta}{1-\theta}<2$, as required.    

The bound for the upper box dimension follows from the continuity of the Assouad spectrum, where
\[
\ubd G_f = \lim_{\theta \to 0} \ad^\theta G_f \leq \lim_{\theta \to 0} 2-\frac{\eta-\theta}{1-\theta} = 2-\eta.
\]
\end{proof}

\begin{proof}[Proof of Lemma \ref{omgCcvLim}]
Let
\[
\eta = \liminf_{t \to 0} \frac{\log \omega(t)}{\log t}
\]
and observe that $\eta \geq 0$. Let $\epsilon>0$ and $T \in (0,1]$; choose $t \in (0,T]$ such that 
\[
\frac{\log \omega(t)}{\log t} \leq \eta + \epsilon.
\]
Let $\alpha \in (0,1)$ so that $\alpha t \in (0,t)$. Then by the concavity of $\omega$,
\[
\frac{\log \omega(\alpha t)}{\log \alpha t} \leq \frac{\log \alpha\omega(t)}{\log\alpha t} \leq \eta+\epsilon + \frac{\log \alpha}{\log \alpha+\log t} \to \eta + \epsilon
\]
as $T \to 0$. Since $\epsilon>0$ was arbitrarily chosen, we conclude that
\[
\limsup_{t \to 0} \frac{\log \omega(\alpha t)}{\log \alpha t} \leq \eta
\]
for all $\alpha \in (0,1)$. This proves the claim.
\end{proof}

We prove Theorem \ref{specTyp} in two parts: first we show the desired set is dense, then we check it is $G_\delta$. 

For a function $f:[0,1] \to \bb{R}$ and an interval $I \subseteq [0,1]$, we refer to
\[
\sup_{x, y \in I} |f(x)-f(y)|
\]
as the \emph{oscillation of $f$ over $I$}.

\begin{lma} \label{omgSpecDense}
Let $\omega$ be a concave modulus of continuity where the limit
\[
\eta = \lim_{t \to 0} \frac{\log\omega(t)}{\log t}
\]
satisfies $\eta < 1$. Then for $\theta < \eta$, the set
\[
\bigg\{f \in \lilmega \ | \ \ad^{\theta} G_f = 2- \frac{\eta-\theta}{1-\theta}\bigg\}
\]
is dense in $(\lilmega, ||\cdot||_\omega)$.
\end{lma}
\begin{proof}
The limit $\eta$ exists by Lemma \ref{omgCcvLim}. Let $f \in \lilmega$, $\theta < \eta$ and $\epsilon>0$. We may assume $\epsilon$ is small enough that $\omega(\epsilon) \leq 1$, so that $\omega(t) \leq 1$ for all $t \leq \epsilon$. Our aim is to show that there exists $\tilf \in \lilmega$ with $||f-\tilf||_\omega \leq \epsilon$ and $\ad^{\theta} G_{\tilf} \geq 2- \frac{\eta-\theta}{1-\theta}$, then we will use Proposition \ref{omgBound} for the reverse inequality.

Let $\delta >0$. Let $(\epsilon_k)$ be a decreasing sequence of positive numbers with $\sum_k \epsilon_k \leq 1/4$ and $\epsilon_k \leq 28^{-1}\epsilon$ for all $k$. Let $(m_k)$ be a sequence of odd positive integers where $m_k \geq (\epk/2)^{(\theta-1)/\theta}$ with the following properties. Write $r_k = m_k^{1/(\theta-1)}$ and $L_k = 2r_k^\theta$, notice that both $r_k, L_k \searrow 0$ with $L_k \leq \epk$. Since $\log\omega(t)/\log(t) \to \eta > \theta$ as $t \to 0$, we may assume each $m_k$ to be large enough so that 
\begin{equation} \label{omgSpecDenseEq}
\theta \leq \frac{\log\omega(r_k)+\log\epk}{\log r_k} \leq \eta+\delta, 
\end{equation}
which rearranges to
\[
r_k^{\eta+\delta}\leq \epk\omega(r_k) \leq r_k^\theta.
\]
Additionally, because $f \in \lilmega$ we may choose $m_k$ such that if $x, y \in [0,1]$ are distinct with $|x-y| \leq L_k$, then
\begin{equation} \label{cond41}
\Fo{x}{y} \leq \epsilon_k.   
\end{equation}
It follows that for such $x, y,$
\begin{equation} \label{cond42}
|f(x)-f(y)| \leq \epk\omega(L_k) \leq \epsilon_k.    
\end{equation}

Consider the intervals $I_1 = [0, L_1]$ and $I_k = [2(L_1+...+L_{k-1}), 2(L_1+...+L_{k-1})+L_k]$ for $k \geq 2$. Observe that the $I_k$ are pairwise disjoint and $\cup_k I_k \subseteq [0,1]$. Denote the left endpoint of $I_k$ by $y_k$ and the right endpoint by $z_k$. We define a new function $\tilde{f}$ by modifying $f$ over $I_k$ so that $\tilf$ coincides with $f$ on $[0,1] \setminus \cup_k I_k$.

Let $k \in \bb{N}$ and assume without loss of generality that $f(y_k) \geq f(z_k)$. Let $y_k'= (y_k+z_k)/2$ and $x_k = (y_k+y_k')/2$. Define $\tilde{f}$ as follows. Let $\tilde{f}(y_k) = f(y_k)$, $\tilde{f}(z_k)=f(z_k)$, $\tilde{f}(x_k) = f(y_k)-\epk\omega(r_k)/2$, $\tilde{f}(y_k') = f(y_k)-\epk\omega(r_k)$ and interpolate linearly between $y_k'$ and $z_k$. Now $|y_k-y_k'|= r_k^\theta = m_k^{\theta/(\theta-1)}$, thus we divide $[y_k,y_k']$ into $r_k^{\theta-1}=m_k$ sub-intervals of length $r_k$. Furthermore, let $\tilde{f}$ be such that:
\begin{itemize}
\item $\tilf$ is continuous;

\item $\tilde{f}$ is affine over each sub-interval;

\item the oscillation of $\tilde{f}$ over each sub-interval equals $\epk\omega(r_k)$;

\item $\tilde{f}(x) \in [\tilde{f}(y_k), \tilde{f}(y_k')]$ for all $x \in [y_k,y_k']$.
\end{itemize}

\begin{figure}[ht]
\centering
\includegraphics[width=0.66\textwidth]{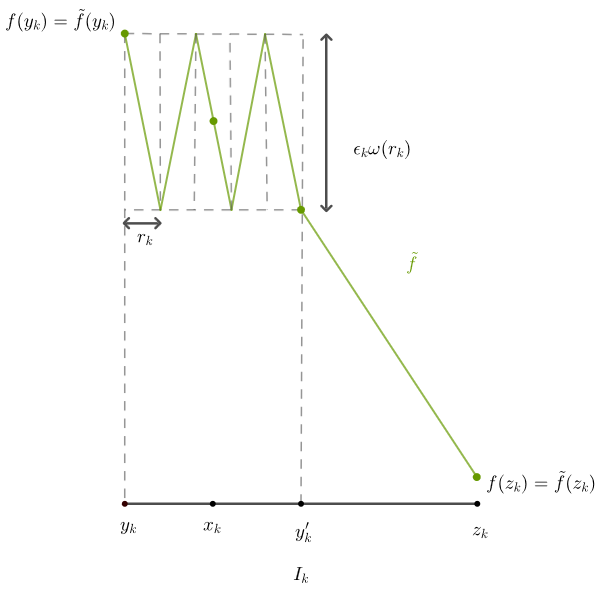}
\caption{$\tilf$ over $I_k$} 
\label{omegaFig}
\end{figure}

see Figure \ref{omegaFig}. Now for distinct $x,y \in [y_k,y_k']$, if $|x-y| > r_k$ then
\[
\tFo{x}{y} \leq \frac{\epk\omega(r_k)}{\omega(r_k)} = \epk;
\]
otherwise, $|x-y| \leq r_k$ and $x,y$ belong to either the same sub-interval of length $r_k$ or two neighbouring ones. In the former case, Lemma \ref{linIneq2} implies
\[
\tFo{x}{y} \leq \frac{\epk\omega(r_k)}{\omega(r_k)} = \epk;
\]
whereas for the latter case, denote by $z$ the shared endpoint of the two neighbouring sub-intervals and assume $x<z<y$. Then since $\omega$ is increasing and $|x-y|= |x-z|+|z-y|$, we have
\[
\max\{\omega(|x-z|), \omega(|z-y|)\} \leq \omega(|x-y|)
\]
and Lemma \ref{linIneq2} again yields
\begin{align*}
\tFo{x}{y} &\leq \frac{|\tilf(x)-\tilf(z)|+ |\tilf(z)-\tilf(y)|}{\omega(|x-y|)} \\ 
&\leq \tFo{x}{z} + \tFo{z}{y} \\
& \leq \frac{\epk\omega(r_k)}{\omega(r_k)} + \frac{\epk\omega(r_k)}{\omega(r_k)}  \\
& = 2\epk.
\end{align*} 
Note that $|y_k-z_k|=L_k$, so by (\ref{cond42})
\[
|\tilf(y_k')-\tilf(z_k)| \leq |f(y_k)-f(z_k)| + \epk\omega(r_k) \leq \epsilon_k \omega(L_k) + \epk\omega(r_k).
\]
Lemma \ref{linIneq2} then shows that for $x, y \in [y_k',z_k]$ distinct
\begin{align*}
\tFo{x}{y} &\leq \tFo{y_k'}{z_k} \\
&\leq \frac{\epsilon_k \omega(L_k) + \epk\omega(r_k)}{\omega\big(\frac{L_k}{2}\big)}  \\
&\leq \frac{2\epk \omega(L_k)}{\omega(L_k)} + \frac{2\epk \omega(L_k)}{\omega(L_k)} \\
&\leq 2\epk + 2\epk \\
&=4\epk,
\end{align*}
where we used that $\omega$ is concave and increasing. Similarly if $x,y \in I_k$ with $x < y_k'<y$, using the same argument we see that
\begin{align*}
\tFo{x}{y} &\leq \frac{|\tilf(x)-\tilf(y_k')|+ |\tilf(y_k')-\tilf(y)|}{\omega(|x-y|)} \\ 
&\leq \tFo{x}{y_k'} + \tFo{y_k'}{y} \\
& \leq 2\epk + 4\epk  \\
& = 6\epk.
\end{align*}
Hence for $x,y \in I_k$
\begin{equation} \label{cond43}
\tFo{x}{y} \leq  6\epsilon_k,
\end{equation}
and since we assumed $\omega(\epsilon) \leq 1$,
\begin{equation} \label{cond44}
|\tilf(x) - \tilf(y)| \leq 6\epsilon_k \omega(|x-y|) \leq 6\epsilon_k \omega(\epsilon) \leq  6\epsilon_k.    
\end{equation}

We check $||f-\tilde{f}||_\omega \leq \epsilon$. Let $x \in [0,1]$ and suppose $f(x) \neq \tilde{f}(x)$, so we must have $x \in I_k$ for some $k$. Then
\[
|f(x) - \tilf(x)| \leq |f(x) - f(z_k)| + |\tilf(z_k) - \tilf(x)| \leq \epsilon_k + 6\epsilon_k \leq \frac{\epsilon}{2}
\]
by (\ref{cond42}) and (\ref{cond44}). Hence $||f-\tilde{f}||_\infty \leq \epsilon/2$. For $[f-\tilf]_\omega$, write $h = f-\tilf$. Let $x, y \in [0,1]$, where we suppose $f(x) \neq \tilde{f}(x)$ and $f(y) \neq \tilde{f}(y)$. Then there must exist $k,l \in \bb{N}$ such that $x \in I_k$ and $y \in I_l$. Assume further that $k<l$. Then by (\ref{cond41}) and (\ref{cond43})
\begin{align*}
\Ho{x}{y} &= \frac{|h(x)- h(z_k) + h(y_l)-h(y)|}{\omega(|x-y|)} \\
& \leq \frac{|f(x) - f(z_k)|}{\omega(|x-y|)} + \frac{|\tilf(z_k) - \tilf(x)|}{\omega(|x-y|)} +  \frac{|f(y_l) - f(y)|}{\omega(|x-y|)} +  \frac{|\tilf(y) - \tilf(y_l)|}{\omega(|x-y|)}  \\
& \leq \Fo{x}{z_k} + \tFo{z_k}{x} + \Fo{y_l}{y} + \tFo{y}{y_l}\\
& \leq \epk + 6\epk + \epsilon_l + 6\epsilon_l \\
& \leq \frac{\epsilon}{2},
\end{align*}
since $h(z_k)=h(y_l)=0$. Other cases where e.g., $k=l$ or $\tilf(x)=f(x)$ are similar. Hence $[f-\tilf]_\omega= [h]_\omega \leq \epsilon/2$ and 
\[
||f-\tilf||_\omega= ||f-\tilf||_\infty + [f-\tilf]_\omega \leq \frac{\epsilon}{2}+\frac{\epsilon}{2}=\epsilon
\]
as desired. 

Next we check $\tilf \in \lilmega$. Let $t \in (0, L_1]$ and define $m$ to be the largest integer such that $t \leq L_m$. Let $x,y \in [0,1]$ with $x<y$ and $|x-y| \leq t$. Suppose $f(x) \neq \tilde{f}(x)$ and $f(y) \neq \tilde{f}(y)$. There must exist $k,l \in \bb{N}$ such that $x \in I_k$ and $y \in I_l$, and we suppose further that $k<l$. Since $L_k \leq |x-y| \leq L_m$, we must have $k \geq m$. Again by (\ref{cond41}) and (\ref{cond43})
\begin{align*}
\tFo{x}{y} &\leq \frac{|\tilf(x)-\tilf(z_k)|+|\tilf(z_k)-\tilf(y_l)|+|\tilf(y_l)-\tilf(y)|}{\omega(|x-y|)} \\
&\leq \tFo{x}{z_k} + \tFo{z_k}{y_l} + \tFo{y_l}{y} \\
&\leq 6\epk + \epsilon_m + 6\epsilon_l \\
&\leq 13\epsilon_m \\
&\to 0
\end{align*}
as $t \to 0$. Other cases follow similarly (and are in fact simpler) and we conclude that $\tilf \in \lilmega$. 

Finally we compute $\ad^{\theta}G_{\tilf}$. Write $\bm{x_k}=\fx{\tilf}{x_k}$ and let $Q_k = Q(\bm{x_k},r_k^{\theta})$. Then each $Q_k$ intersects at least
\[
\frac{1}{2} \cdot \frac{r_k^{\theta}}{r_k}
\]
columns of the form $[nr_k,(n+1)r_k] \times \bb{R}$ for some $n \in \bb{Z}$ in the $r_k$-mesh grid. By construction, in each such column $G_{\tilf} \cap Q_k$ intersects at least
\[
\frac{1}{2} \cdot\frac{\epk\omega(r_k)}{r_k} 
\]
squares of side-length $r_k$ in the mesh grid, since $\epk\omega(r_k) \leq r_k^{\theta}$. Summing these up,
\begin{align*}
N_{r_k}(Q_k \cap G_{\tilf}) &\geqs \frac{r_k^\theta}{r_k} \cdot \frac{\epk\omega(r_k)}{r_k} \\
& \geq r_k^{\theta-1+\eta+\delta-1} \\
&= r_k^{(\theta-1)\big(2-\frac{\eta-\theta+\delta}{1-\theta}\big)}.
\end{align*}
Since $r_k^{\theta-1} \to \infty$ as $k \to \infty$, we get
\[
\ad^{\theta} G_{\tilf} \geq 2-\frac{\eta-\theta+\delta}{1-\theta}.
\]
Our choice of $\delta >0$ was arbitrary, so we can in fact conclude
\[
\ad^{\theta} G_{\tilf} \geq 2-\frac{\eta-\theta}{1-\theta}.
\]

For the upper bound, observe that $\omega$ being concave implies $\omega(1) < \infty$. Hence Proposition \ref{omgBound} applies if $\theta < \eta$ and $\ad^{\theta} G_{\tilf} \leq 2- \frac{\eta-\theta}{1-\theta}$, thus proving the desired result.
\end{proof}

\begin{lma} \label{specGdelta}
Let $\theta \in (0,1)$ and let $\omega$ be a modulus of continuity. Set
\[
\beta = \max_{f \in \lilmega} \ad^\theta G_f.
\]
Then the set $\{f \in \lilmega \ | \ \ad^{\theta} G_f = \beta \}$ is $G_\delta$.
\end{lma}
\begin{proof}
Let $A$ be the set in the statement. Then since $\ad^{\theta} G_f \leq \beta$ for all $f \in \lilmega$, we can write
\begin{align*}
A &= \bigcap_{\epsilon >0} \bigcap_{C\geq0}\bigcup_{x \in [0,1]} \bigcup_{r>0} \bigg\{f \in \lilmega : N_r(B^o((x,f(x)),r^\theta) \cap G_f) \geq C\bigg(\frac{r^\theta}{r}\bigg)^{\beta-\epsilon} \bigg\} \\ 
&= \bigcap_{n \in \bb{N}} \bigcap_{C \in \bb{N}_0}\bigcup_{x \in [0,1]} \bigcup_{r>0} \bigg\{f \in \lilmega : N_r(B^o((x,f(x)),r^\theta) \cap G_f) \geq C\bigg(\frac{r^\theta}{r}\bigg)^{\beta-\frac{1}{n}} \bigg\},
\end{align*}
so we need only to show the inner most set in the above line is open. To this end, let $n \in \bb{N}$ and $C \in \bb{N}_0$ be given. Suppose $f \in \lilmega$ satisfies
\[
N_r(B^o((x,f(x)),r^\theta) \cap G_f) \geq C\bigg(\frac{r^\theta}{r}\bigg)^{\beta-\frac{1}{n}}
\]
for some $x \in [0,1]$ and $r>0$. Let $\{B(\fx{f}{x_i},r)\}_{i \in I}$ be a maximal $r$-packing of $B^o(\fx{f}{x},r^\theta) \cap G_f$, where $I$ is some finite index set. For each $i$, let $h_i>0$ be the minimum vertical distance of $\fx{f}{x_i}$ to the boundary of $B^o(\fx{f}{x},r^\theta)$ and let $h = \min_{i \in I} h_i$. Moreover, let
\[
d = \min_{i \neq j} |\fx{f}{x_i}-\fx{f}{x_j}|-2r\
\]
and notice that $d>0$. Then taking 
\[
\xi = \frac{1}{2}\min\{ h, d \},
\]
we see that if $g \in \lilmega$ satisfies $||f-g||_{\infty} < \xi$ then $(x_i,g(x_i)) \in B^o((x,g(x)),r^\theta) \cap G_g$ for all $i$ with the collection $P =\{B((x_i,g(x_i)),r)\}_{i \in I}$ disjoint. In other words, $P$ is an $r$-packing of $B^o((x,g(x)),r^\theta) \cap G_g$. Hence
\[
N_r(B^o((x,g(x)),r^\theta) \cap G_g) \geq N_r(B^o(\fx{f}{x},r^\theta) \cap G_f) \geq C\bigg(\frac{r^\theta}{r}\bigg)^{\beta-\frac{1}{n}}.
\]
Since $g \in B^o(f, \xi)$ only if $||f-g||_{\infty} < \xi$, it follows that
\[
B^o(f, \xi) \subseteq \bigg\{f_0 \in \lilmega : N_r(B^o(\fx{f_0}{x},r^\theta) \cap G_{f_0}) \geq C\bigg(\frac{r^\theta}{r}\bigg)^{\beta-\frac{1}{n}} \bigg\},
\]
which proves $A$ is a $G_\delta$ set. 
\end{proof}

\begin{proof}[Proof of Theorem \ref{specTyp}]
Again the limit $\eta$ exists by Lemma \ref{omgCcvLim}. For part (i), Proposition \ref{omgCnst} shows that if $\eta>1$ then $\cmega$ consists of only constant functions, whose graphs have Assouad dimension and spectra equal to $1$; so we assume $\eta = 1$. Let $f \in \cmega$ and $\theta \in (0,1)$, Proposition \ref{omgBound} implies $\ad^\theta G_f \leq 1$. On the other hand, let $0<r<1$ and consider $Q(\bm{0},r^\theta)$. Since $f$ is continuous,
\[
N_r'(Q(\bm{0},r^\theta) \cap G_f) \geq \floor{\frac{r^\theta}{2r}} \geqs \frac{r^\theta}{r}
\]
which gives $\ad^\theta G_f \geq 1$. Thus, $\ad^\theta G_f =1$. 

For part (ii), Proposition \ref{omgBound} and Lemma \ref{omgSpecDense} imply that for all $\theta < \eta$, the sets
\[
F_\theta := \bigg\{f \in \lilmega \ | \ \ad^{\theta} G_f = 2- \frac{\eta-\theta}{1-\theta}\bigg\}
\]
are dense in $\lilmega$. By continuity of the Assouad spectrum
\[
F_\eta = \{f \in \lilmega \ | \ \ad^{\eta} G_f = 2 \}
\]
is dense in $\lilmega$ where all $f \in F_\eta$ also satisfy $\ad G_f = 2$. Using \cite[Corollary 3.6]{FrYu1} which states that if $\ad^\phi E = \ad E$ then $\ad^{\phi'} E = \ad E$ for all $\phi' \in [\phi,1)$, it follows that $\ad^\theta G_f = \ad G_f= 2$ for all $\theta>\eta$ and $f \in F_\eta$. Hence for $\theta \geq  \eta$, the sets
\[
F_\theta := \{f \in \lilmega \ | \ \ad^{\theta} G_f = 2 \}
\]
all contain $F_\eta$, which means they are dense as well. Lemma \ref{specGdelta} then tells us that $F_\theta$ is residual for all $\theta \in (0,1)$.   One final simple but important step is needed to complete the proof.  So far we have only established that, for each $\theta$, a typical function satisfies the correct formula for $\theta$.  However, what we want to prove is that a typical function satisfies the correct formula for \emph{all $\theta$ simultaneously}.  Since there are an uncountable number of $\theta$ to consider, and the property of being co-meagre is not preserved under uncountable intersection, this could be a problem.  However, this problem is overcome by appealing to continuity of the Assouad spectrum.  In particular, it is determined by its values on a countable dense subset.  Indeed, 
\[
F : = \bigcap_{\theta \in (0,1)} F_\theta = \bigcap_{\theta \in (0,1) \cap \bb{Q}}  F_\theta.
\]
Therefore, $F$ is a countable intersection of residual sets and so is itself residual.  This completes the proof.
\end{proof}

\subsection[\texorpdfstring{Assouad dimension of typical little $\omega$ graphs}
                        {Assouad dimension of typical little omega graphs}]
        {Assouad dimension of typical little $\omega$ graphs} \label{adProofs}

We focus on the Assouad dimension of the graphs and prove Theorem \ref{adTyp} in this section. We start with the Lipschitz functions and write
\[
[f]_1 = \sup_{x,y \in [0,1]} \frac{|f(x)-f(y)|}{|x-y|}.
\]

\begin{prop} \label{lip}
Suppose $f:[0,1] \to \bb{R}$ is Lipschitz. Then $ \ad^\theta G_f= \ad G_f =1$ for all $\theta \in (0,1)$.     
\end{prop}
\begin{proof}
Let  $0<r<R$. Let $x \in [0,1]$ and write $\bm{x} = (x,f(x))$. The number of $r$-columns of the form $[nr,(n+1)r] \times \bb{R}$ for some integer $n$ which $Q(\bm{x},R)$ intersects is at most $R/r+2$. Since 
\[
|f(y) - f(z)| \leq [f]_1|y-z|\leq [f]_1 r
\]
for all $y,z \in [nr,(n+1)r]$, the number of $r$-squares $G_f$ intersects in each $r$-column is at most $[f]_1 r/r +2 = [f]_1+2$. Hence
\[
N_r'(Q(\bm{x},R) \cap G_f) \leq \bigg(\frac{R}{r} + 2\bigg)([f]_1+2) \leqs \frac{R}{r},
\]
from which we conclude $\ad G_f \leq 1$. For the lower bound, notice that for $\theta \in (0,1)$ Theorem \ref{specTyp} (i) implies $\ad^\theta G_f = 1$. Hence, 
\[
1 =\ad^\theta G_f \leq \ad G_f \leq 1
\]
and equality must hold throughout.
\end{proof}

Our proof of Theorem \ref{adTyp} is very similar to that of Theorem \ref{specTyp}, however we include it for completeness. We follow the same strategy: we check the desired set is dense, then show that it is $G_\delta$.

\begin{lma} \label{adDense}
Let $\omega$ be a concave modulus of continuity satisfying 
\[
\lim_{t \to 0} \frac{\omega(t)}{t} = \infty.
\]
Then the set $\{f \in \lilmega \ | \ \ad G_f = 2 \}$ is dense in $(\lilmega, ||\cdot||_\omega)$.
\end{lma}
\begin{proof}
Let $f \in \lilmega$ and $\epsilon>0$. We may assume $\epsilon$ is small enough such that $\omega(\epsilon) \leq 1$, so that $\omega(t) \leq 1$ for all $t \leq \epsilon$. Our aim is to show that there exists $\tilde{f} \in \lilmega$ with $||f-\tilde{f}||_\omega \leq \epsilon$ and $\ad G_{\tilde{f}} = 2$.    

Let $(\epsilon_k)$ be a decreasing sequence of positive real numbers with $\sum_k \epsilon_k \leq 1/8$ and $\epsilon_k \leq 24^{-1}\epsilon$. Let $(r_k)$ be a positive and strictly decreasing sequence where $\omega(r_k) \leq \epk$ and 
\[
\frac{\omega(r_k)}{r_k} \geq \frac{2k+1}{\epk}
\]
for all $k$, using the fact that $\omega$ is increasing and $\omega(t)/t \to \infty$ as $t \to 0$. Additionally, since $f \in \lilmega$ we may choose $r_k$ such that if $x, y \in [0,1]$ are distinct with $|x-y| \leq \omega(r_k)$, then
\begin{equation} \label{cond21}
\Fo{x}{y} \leq \epsilon_k.   
\end{equation}
It follows that for such $x, y$
\begin{equation} \label{cond22}
|f(x)-f(y)| \leq \omega(\epk)\epk \leq \epsilon_k.    
\end{equation}
Finally, we define $R_k = (2k+1)r_k$ and note that $R_k \searrow 0$ with $R_k \leq \omega(r_k)\epk$. 

Consider the intervals $I_1 = [0, 2R_1]$ and $I_k = [4(R_1+...+R_{k-1}), 4(R_1+...+R_{k-1})+2R_k]$ for $k \geq 2$; observe that the $I_k$ are disjoint and $\cup_k I_k \subseteq [0,1]$. Denote the left endpoint of $I_k$ by $y_k$ and the right endpoint by $z_k$. We define a new function $\tilde{f}$ by modifying $f$ over $I_k$ so that $\tilf$ coincides with $f$ on $[0,1] \setminus \cup_k I_k$.

\begin{figure}[ht] 
\centering
\includegraphics[width=0.66\textwidth]{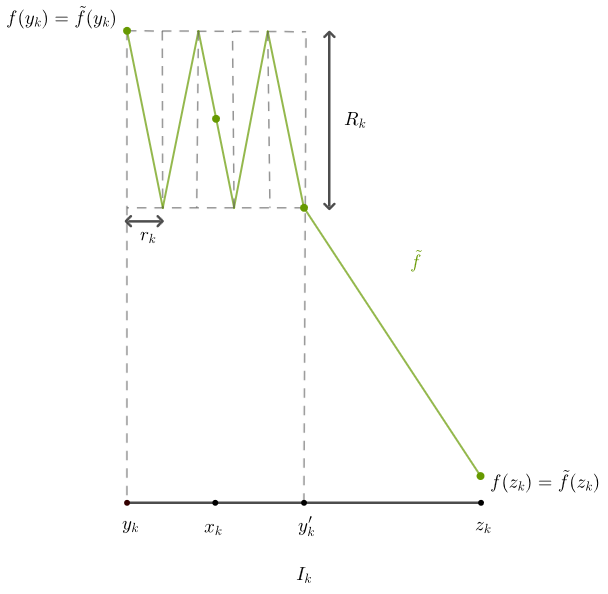}
\caption{$\tilf$ over $I_k$}
\label{adFig}
\end{figure}

Let $k \in \bb{N}$ and assume without loss of generality that $f(y_k) \geq f(z_k)$. Let $y_k'= (y_k+z_k)/2$ and $x_k = (y_k+y_k')/2$. Define $\tilde{f}$ as follows. Let $\tilde{f}(y_k) = f(y_k)$, $\tilde{f}(z_k)=f(z_k)$, $\tilde{f}(x_k) = f(y_k)-R_k/2$, $\tilde{f}(y_k') = f(y_k)-R_k$ and interpolate linearly between $y_k'$ and $z_k$. Since $|y_k'-y_k|=R_k$, we divide $[y_k,y_k']$ into $2k+1$ sub-intervals of length $r_k$. Furthermore, let $\tilde{f}$ be such that:
\begin{itemize}
\item $\tilde{f}$ is affine over each sub-interval;

\item the oscillation of $\tilde{f}$ over each sub-interval equals $R_k$;

\item $\tilde{f}(x) \in [\tilde{f}(y_k), \tilde{f}(y_k')]$ for all $x \in [y_k,y_k']$,
\end{itemize}
see Figure \ref{adFig}. Now for distinct $x,y \in [y_k,y_k']$, if $|x-y| > r_k$ then
\[
\tFo{x}{y} \leq \frac{R_k}{\omega(r_k)} \leq \epk;
\]
otherwise, $|x-y| \leq r_k$ and $x,y$ belong to either the same sub-interval or two neighbouring ones. In the former case, Lemma \ref{linIneq2} implies
\[
\tFo{x}{y} \leq \frac{R_k}{\omega(r_k)} \leq \epk;
\]
whereas for the latter case, denote by $z$ the shared endpoint of the two neighbouring sub-intervals and assume $x<z<y$. Then since $\omega$ is increasing and $|x-y|= |x-z|+|z-y|$, we have
\[
\max\{\omega(|x-z|), \omega(|z-y|)\} \leq \omega(|x-y|)
\]
and Lemma \ref{linIneq2} again yields
\begin{align*}
\tFo{x}{y} &\leq \frac{|\tilf(x)-\tilf(z)|+ |\tilf(z)-\tilf(y)|}{\omega(|x-y|)} \\ 
&\leq \tFo{x}{z} + \tFo{z}{y} \\
& \leq \frac{R_k}{\omega(r_k)} + \frac{R_k}{\omega(r_k)}  \\
& \leq 2\epk.
\end{align*}
Note that $|y_k-z_k|=2R_k \leq \omega(r_k)$, so by (\ref{cond21})
\[
|\tilf(y_k')-\tilf(z_k)| \leq |f(y_k)-f(z_k)| + R_k \leq \epsilon_k \omega(2R_k) + R_k.
\]
Lemma \ref{linIneq2} then shows that for $x, y \in [y_k',z_k]$ distinct
\begin{align*}
\tFo{x}{y} &\leq \tFo{y_k'}{z_k} \\
&\leq \frac{ \epsilon_k \omega(2R_k) + R_k}{\omega(R_k)}  \\
&\leq \frac{2\epk \omega(R_k)}{\omega(R_k)} + \frac{R_k}{\omega(r_k)} \\
&\leq 2\epk + \epk \\
&=3\epk,
\end{align*}
where we used that $\omega$ is concave and increasing.

Similarly if $x,y \in I_k$ with $x < y_k'<y$, using the same argument we see that
\begin{align*}
\tFo{x}{y} &\leq \frac{|\tilf(x)-\tilf(y_k')|+ |\tilf(y_k')-\tilf(y)|}{\omega(|x-y|)} \\ 
&\leq \tFo{x}{y_k'} + \tFo{y_k'}{y} \\
& \leq 2\epk + 3\epk  \\
& = 5\epk.
\end{align*}
Hence for $x,y \in I_k$
\begin{equation} \label{cond23}
\tFo{x}{y} \leq  5\epsilon_k,
\end{equation}
and since we assumed $\omega(\epsilon) \leq 1$
\begin{equation} \label{cond24}
|\tilf(x) - \tilf(y)| \leq 5\epsilon_k \omega(|x-y|) \leq 5\epsilon_k \omega(\epsilon) \leq  5\epsilon_k.    
\end{equation}

We check $||f-\tilde{f}||_\omega \leq \epsilon$. Let $x \in [0,1]$ and suppose $f(x) \neq \tilde{f}(x)$, so we must have $x \in I_k$ for some $k$. Then
\[
|f(x) - \tilf(x)| \leq |f(x) - f(z_k)| + |\tilf(z_k) - \tilf(x)| \leq \epsilon_k + 5\epsilon_k \leq \frac{\epsilon}{2}
\]
by (\ref{cond22}) and (\ref{cond24}). Hence $||f-\tilde{f}||_\infty \leq \epsilon/2$. For $[f-\tilf]_\omega$, write $h = f-\tilf$. Let $x, y \in [0,1]$, where we suppose $f(x) \neq \tilde{f}(x)$ and $f(y) \neq \tilde{f}(y)$. Then there must exist $k,l \in \bb{N}$ such that $x \in I_k$ and $y \in I_l$. Assume further that $k<l$. Then by (\ref{cond21}) and (\ref{cond23})
\begin{align*}
\Ho{x}{y} &= \frac{|h(x)- h(z_k) + h(y_l)-h(y)|}{\omega(|x-y|)} \\
& \leq \frac{|f(x) - f(z_k)|}{\omega(|x-y|)} + \frac{|\tilf(z_k) - \tilf(x)|}{\omega(|x-y|)} +  \frac{|f(y_l) - f(y)|}{\omega(|x-y|)} +  \frac{|\tilf(y) - \tilf(y_l)|}{\omega(|x-y|)}  \\
& \leq \Fo{x}{z_k} + \tFo{z_k}{x} + \Fo{y_l}{y} + \tFo{y}{y_l}\\
& \leq \epk + 5\epk + \epsilon_l + 5\epsilon_l \\
& \leq \frac{\epsilon}{2},
\end{align*}
since $h(z_k)=h(y_l)=0$. Other cases where, for example, $k=l$ or $\tilf(x)=f(x)$ are similar. Hence $[f-\tilf]_\omega= [h]_\omega \leq \epsilon/2$ and 
\[
||f-\tilf||_\omega= ||f-\tilf||_\infty + [f-\tilf]_\omega \leq \frac{\epsilon}{2}+\frac{\epsilon}{2}=\epsilon
\]
as desired. 

Next we check $\tilf \in \lilmega$. Let $t \in (0,L_1]$ and define $m$ to be the largest integer such that $t \leq 2R_m$. Let $x,y \in [0,1]$ with $x<y$ and $|x-y| \leq t$. Suppose $f(x) \neq \tilde{f}(x)$ and $f(y) \neq \tilde{f}(y)$. There must exist $k,l \in \bb{N}$ such that $x \in I_k$ and $y \in I_l$, and we suppose further that $k<l$. Since $2R_k \leq |x-y| \leq 2R_m$, we must have $k \geq m$. Again by (\ref{cond21}) and (\ref{cond23})
\begin{align*}
\tFo{x}{y} &\leq \frac{|\tilf(x)-\tilf(z_k)|+|\tilf(z_k)-\tilf(y_l)|+|\tilf(y_l)-\tilf(y)|}{\omega(|x-y|)} \\
&\leq \tFo{x}{z_k} + \tFo{z_k}{y_l} + \tFo{y_l}{y} \\
&\leq 5\epk + \epsilon_m + 5\epsilon_l \\
&\leq 11\epsilon_m \\
&\to 0
\end{align*}
as $t \to 0$. Other cases follow similarly (and are in fact simpler) and we conclude that $\tilf \in \lilmega$. 

Finally we compute $\ad G_{\tilf}$. Write $\bm{x_k}=\fx{\tilf}{x_k}$ and let $Q_k = Q(\bm{x_k},R_k)$. Then each $Q_k$ intersects at least
\[
\frac{1}{2} \cdot \frac{R_k}{r_k}
\]
columns of the form $[nr_k,(n+1)r_k] \times \bb{R}$ for some $n \in \bb{Z}$ in the $r_k$-mesh grid. By construction, in each such column $G_{\tilf} \cap Q_k$ intersects at least
\[
\frac{1}{2} \cdot\frac{R_k}{r_k} 
\]
squares of side-length $r_k$ in the mesh grid. Summing these up,
\[
N_{r_k}(Q_k \cap G_{\tilf}) \geqs \bigg(\frac{R_k}{r_k}\bigg)^2.
\]
Since
\[
\frac{R_k}{r_k} = 2k+1 \to \infty
\]
as $k \to \infty$, we get $\ad G_{\tilf} = 2$, which concludes the proof.
\end{proof}

\begin{lma} \label{adGelta}
Let $\omega$ be a modulus of continuity. The set $\{f \in \lilmega \ | \ \ad G_f = 2 \}$ is $G_\delta$.
\end{lma}
\begin{proof}[Proof sketch]
We consider the set
\[
\bigcap_{\epsilon >0} \bigcap_{C\geq0}\bigcup_{x \in [0,1]} \bigcup_{R>0} \bigcup_{r \in (0,R)} \bigg\{f \in \lilmega : N_r(B^o((x,f(x)),R) \cap G_f) \geq C\bigg(\frac{R}{r}\bigg)^{2-\epsilon} \bigg\},
\]
which can be shown to be $G_\delta$ in a similar way as in the proof of Lemma \ref{specGdelta}.
\end{proof}

\begin{proof}[Proof of Theorem \ref{adTyp}]
Part (i) follows from Proposition \ref{lip}. Lemma \ref{adDense} and \ref{adGelta} yield the conclusion in part (ii).
\end{proof}

\bibliographystyle{abbrv}
\bibliography{bib.bib}

\end{document}